\documentclass[11pt]{article}
\usepackage[english]{babel}
\usepackage{amssymb,amsmath,amsthm}
\usepackage[dvipsnames]{pstricks}
\newtheorem{theorem}{Theorem}[section]
\newtheorem{lemma}[theorem]{Lemma}

\newtheorem{corollary}[theorem]{Corollary}

\newtheorem{example}[theorem]{Example}
\newtheorem{question}[theorem]{Question}

{\theoremstyle{definition}}
{\theoremstyle{definition}\newtheorem{definition}[theorem]{Definition}}
{\theoremstyle{definition}\newtheorem{remark}[theorem]{Remark}}

\numberwithin{equation}{section}

\def\C{{\mathbb C}}
\def\N{{\mathbb N}}
\def\Z{{\mathbb Z}}

\def\Q{{\mathbb Q}}
\def\K{{\mathbb K}}
\def\F{{\mathbb F}}
\def\aaa{{\mathbb A}}
\def\epsilon{\varepsilon}
\def\kappa{\varkappa}
\def\phi{\varphi}
\def\leq{\leqslant}
\def\geq{\geqslant}

\def\dim{{\rm dim}\,}

\def\deg{\hbox{\tt deg}\,}

\title{Quadratic automaton algebras and intermediate growth}

\author{Natalia Iyudu and Stanislav Shkarin}

\date{}

\begin{document}

\maketitle

\begin{abstract} We present an example of a quadratic algebra given by three generators and three relations, which is automaton (the set of normal words forms a regular language) and such that its ideal of relations does not possess a finite Gr\"obner basis with respect to any choice of generators and any choice of a well-ordering of monomials compatible with multiplication. This answers a question of Ufnarovski.

 Another result is a simple example (4 generators and 7 relations) of a quadratic algebra of intermediate growth.
\end{abstract}

\small \noindent{\bf MSC:} \ \ 17A45, 16A22

\noindent{\bf Keywords:} \ \ Automaton algebras, Word combinatorics, Regular Language, Gr\"obner basis, Quadratic algebras, Koszul algebras, Finitely presented algebras, Hilbert series, Gelfand--Kirillov dimension   \normalsize

\section{Introduction \label{s1}}\rm

In the first part of the paper we consider automaton algebras, that is algebras which could be associated to a finite graph.
This is a property of an algebra, formulated in terms of language of its normal words, that is in terms of words combinatorics. This approach to the study of algebra (its linear basis, Hilbert series, growth, structural properties) is very powerful, widely known,
see, for example, book of collective author M.~Lothaire
 \cite{L}  or A.~Belov, V.~Borisenko, V.~Latyshev \cite{BBL}, or book of M.~Sapir \cite{S}, and fruitfully connects algebraic study with symbolic dynamics.

The class of automaton algebras was introduced by
Ufnarovski \cite{ufn2} (see also his book  \cite[Chapter~5]{ufn1}).
He asked if every automaton algebra admits a finite Gr\"obner basis in its ideal of relations with respect to some choice of generators and an ordering on monomials.
It is a very natural question, and many people tried to figure it out, however without a success.
The difficulty here lies in the fact that there are infinitely many generating sets as well as infinitely many admissible orderings on monomials to deal with.

We argue, using ${\rm mod}\,p$ specialization of coefficients of initial algebra over the field of characteristic zero. On the way we prove quite interesting statement, allowing to capture the fact of existence  of finite Gr\"obner basis. It says that if algebra over a field of characteristic zero did have a finite Gr\"obner bases, then after specialization ${\rm mod}\,p$ its growth can't go up (thus stays the same) for all but finitely many p. Then the difficulty of considering all possible orderings transforms to consideration of all possible ${\rm mod}\,p$ specializations, and this turns out to be more tractable, we are able to follow the the pattern of the Gr\"obner basis for any $p$. We prove that in a particular example, namely, the algebra with three generators $x,y,z$,
and relations

$$
xx-xz-2zz=0,\ yz=0,\ xy=0,
$$
which is automaton and has a polynomial growth, Hilbert series $H(t)= \frac{1}{(1-t)^3}$ in characteristic zero and two, in any odd finite characteristic becomes an algebra of exponential growth.

Thus, the first main result of this paper is an example of a quadratic automaton algebra $A$ given by three generators and three relations such that the ideal of relations of $A$ does not possess a finite Gr\"obner basis with respect to any choice of generators and any choice of a well-ordering on monomials compatible with multiplication. A finite Gr\"obner basis does not exist after any  extension of the ground field as well.

The second our result  is a simple example of quadratic algebra (4 generators, 7 relations)  with intermediate growth.

 One example of finitely presented cubic algebra of intermediate growth was mentioned in a note added in proof of the paper by  Shearer \cite{she}. There was no explanation why the Hilbert series is as stated, and it was elaborated in paper by Kobayashi \cite{koba}, where he also constructs other examples and applies them to the word problem in monoids.
  Another example was due to Ufnarovskii \cite{ufn3}, who proves that the universal enveloping of the (infinite dimensional) Witt Lie algebra is finitely presented and has intermediate growth. The paper \cite{kkm} by Kirillov, Kontsevich and Molev contains a number of results on algebras of intermediate growth both in the associative and Lie setting together with yet another example of universal enveloping with intermediate growth. All these examples were not quadratic.
  Recently,  Ko\c cak \cite{ko} used the idea suggested by Zelmanov, that one can take
  a Veronese subalgebra of (a slightly modified to make it degree graded) universal enveloping of the Witt algebra, to obtain quadratic algebra with intermediate growth.
  One drawback of his example is its size: it is presented by 14 generators and 96 quadratic relations. Literally, the presentation of this quadratic algebras in \cite{ko} has two and a half pages of relations.
So we thought, that example of quadratic algebra with 4 generators and  7 relations,  with intermediate growth, could be of interest.

Now we formulate results more precisely.
We start by recalling relevant definitions. All algebras we deal with are associative. Throughout the paper, $\K$ is a field. If $A$ is a unital $\K$-algebra and $X\subset A$ generates $A$ (as a unital algebra), then $A$ can be naturally interpreted as the quotient of the free algebra $\K\langle X\rangle$ on generators $X$ by the ideal $I$ consisting of all $f\in \K\langle X\rangle$ vanishing when considered as elements of $A$. The ideal $I$ is known as the \textbf{\textit{ideal of relations}} of $A$. A generating set $X$ together with a set $F\subseteq I$ generating $I$ as a (two-sided) ideal is known as a \textbf{\textit{presentation}} of $A$. An algebra is called \textbf{\textit{finitely presented}} if $X$ and $F$ can be chosen finite. In other words, a finitely presented algebra is an algebra isomorphic to a quotient of $\K\langle X\rangle$ with finite $X$ by a finitely generated ideal. The set of all words (including the empty word $1$) in the alphabet $X$ is denoted $\langle X\rangle$. A well-ordering $\leq$ on $\langle X\rangle$ is said to be \textbf{\textit{compatible with multiplication}} if
$$
\text{$1\leq u$ for all $u\in \langle X\rangle$ and $u\leq v\,\Longrightarrow\,uw\leq vw,\ wu\leq wv$ for all $u,v,w\in \langle X\rangle$.}
$$
If we fix a well-ordering $\leq$ on $\langle X\rangle$ compatible with multiplication, we can talk of the leading monomial $\overline{f}$ of a non-zero $f\in \K\langle X\rangle$ (=the biggest with respect to $\leq$ monomial, which features in $f$ with non-zero coefficient). A subset $G$ of an ideal $I$ in $\K\langle X\rangle$ is called a \textbf{\textit{Gr\"obner basis of $I$}} if $0\notin G$, $G$ generates $I$ as an ideal and for each non-zero $f\in I$, there is $g\in G$ such that $\overline g$ is a subword of $\overline{f}$. Such a subset is by no means unique: for one, $I\setminus \{0\}$ fits the bill. However, a couple of extra conditions pinpoint $G$. Namely, if we additionally assume that a Gr\"obner basis $G$ satisfies
\begin{itemize}\itemsep-3pt
\item for every two distinct $f,g\in G$, $\overline f$ is not a subword of any monomial featuring in $g$;
\item every $f\in G$ is monic: the $\overline f$-coefficient in $f$ equals $1$,
\end{itemize}
then $G$ becomes unique. Such a basis is called the \textbf{\textit{reduced Gr\"obner basis of $I$}}. Note that $I$ possesses a finite Gr\"obner basis if and only if its reduced Gr\"obner basis is finite.

The non-commutative Buchberger algorithm \cite{ber} applied to the set of defining relations yields the reduced Gr\"obner basis for the ideal of relations of any finitely presented algebra. One of the problems though is that (unlike for the commutative case) the procedure does not have to terminate in finitely many steps. What is even worse, there is no {\it a-priory} way to say if it does. Furthermore, everything is highly sensitive to the choice of the generators and the ordering. The words $u\in \langle X\rangle$, which have no leading monomials of elements of the ideal $I$ of relations of $A$ as subwords, are called \textbf{\textit{normal words}} for $A$. It is easy to see that normal words form a linear basis in $A$. Clearly, if $G$ is a Gr\"obner basis for $I$, then a word $u\in \langle X\rangle$ is normal if and only if it has no leading monomials of elements of $G$ as subwords.

Recall that if $X$ is a finite set, then a subset $L$ of $\langle X\rangle$ forms a \textbf{\textit{regular language}} if there is a finite quiver (loops and multiple edges are allowed) with each arrow marked by a letter from $X$, while some vertices are marked by symbols $S$ and/or $T$ such that $L$ consists (exactly) of the words that can be read from paths in this quiver starting at a vertex marked by $S$ and terminating at a vertex marked by $T$. Ufnarovski \cite{ufn1} has introduced the class of automaton algebras in the following way.

\begin{definition}\label{aual}
A finitely generated algebra $A$ is called an \textbf{\textit{automaton algebra}} if there exists a finite generating set $X$ for $A$ such that with respect to some well-ordering on $\langle X\rangle$ compatible with multiplication, the set $NW$ of normal words for $A$ forms a regular language.
\end{definition}

There is correspondence between the set of normal words $NW$, and the set $LM$ of leading monomials of the reduced Gr\"obner basis of the ideal of relations of $A$, namely each of them can be reconstructed from another. Namely, $NW$ is the set of all monomials which does not contain elements of $LM$, as subwords, and $LM$, are all those words in which all subwords are from $NW$, but they themselves are not.
 Hence $NW$ is a regular language if and only if the set $LM$  is a regular language.
This means that $NW$ and $LM$ are interchangeable in the definition of automaton algebras. In particular, an algebra, whose ideal of relations possesses a finite Gr\"obner basis (with respect to some choice of a generating set and an ordering on monomials), is an automaton algebra. On the other hand, there are examples \cite{ufn1,MN} of automaton algebras with infinite Gr\"obner bases. The question mentioned above is whether this may be rectified by a clever choice of generators and of an ordering.

\begin{question}\label{mainq}
Given an automaton algebra $A$, is it possible to find a finite generating set $X$ and an ordering on $\langle X\rangle$ compatible with multiplication such that the ideal of relations of $A$ has a finite Gr\"obner basis?
\end{question}


Theoretically speaking, there is one more degree of freedom here. Replacing the ground field $\K$ by a bigger field (a one containing $\K$ as a subfield) preserves the automaton property, but it yields extra opportunities for a finite Gr\"obner basis to exist. Replacing $\K$ by a bigger field $\F$ amounts to replacing a finitely presented $\K$-algebra $A$ by $A_\F=\F\otimes_\K A$ treated as an $\F$-algebra. Clearly, any presentation of $A$ as a $\K$-algebra works as a presentation of $A_\F$ as an $\F$-algebra (but not vice versa).

Note that automaton algebras share a lot of properties with algebras possessing a finite Gr\"obner basis in the ideal of relations \cite{ufn1,ufn2,MN,MN1,aut1}. For instance, the Hilbert series of each automaton algebra is rational (hence intermediate growth is impossible), nil automaton algebras are finite dimensional, an automaton algebra of exponential growth contains a subalgebra isomorphic to the free algebra on two generators, etc. This was another evidence for the conjecture that these two classes coincide.
It is also quite clear, that not every quadratic algebra is automaton. This could be deduced, for instance, from the example
by Shearer \cite{she} of a quadratic algebra (=a finitely presented algebra with all defining relations being homogeneous of degree 2), whose Hilbert series is not rational. Shearer produced this example answering a well-known question of Govorov \cite{gov} whether every finitely presented algebra has rational Hilbert series.

\begin{definition}\label{nogb} We say that a finitely presented $\K$-algebra $A$ is a \textbf{\textit{finite Gr\"obner basis algebra}} (\textbf{\textit{FGB-algebra}} for short), if there exists a field $\F$ containing $\K$ as a subfield such that there is a finite Gr\"obner basis in the ideal of relations of $A_\F$ for some choice of a finite generating set $X$ and of a well-ordering on $\langle X\rangle$ compatible with multiplication. Otherwise, we say that $A$ is an \textbf{\textit{infinite Gr\"obner basis algebra}} (\textbf{\textit{IGB-algebra}} for short).
\end{definition}

Our main objective is to provide examples of finitely presented algebras, which are IGB and automaton. Our method allows to construct unlimited number of such examples. We give two: one quadratic on three generators given by three relations and one cubic on two generators given by two relations.

\begin{theorem} \label{main} Let $\K$ be an arbitrary field of characteristic zero. Let $A$ be the $\K$-algebra given by generators $x,y,z$ and relations $xy$, $yz$ and $x^2-xz-2z^2$ and $B$ be the $\K$-algebra given by generators $x,y$ and relations $y^3$ and $x^2y-yx^2-yxy$. Then both $A$ and $B$ are automaton IGB-algebras.
\end{theorem}

The advantage of $A$ is in being quadratic, while the advantage of $B$ is that it is generated by just two elements. The proof is based on the fact
(Theorem~\ref{nofgb}) that the existence of finite Gr\"obner basis implies that the growth of algebra and any its specialization ${\rm mod}\,p$ coincide (for almost all $p$).
Then the example of algebra without finite  Gr\"obner basis, is an algebra with integer coefficients  and finite Gelfand-Kirillov dimension, which becomes infinite when coefficients are written ${\rm mod}\,p$, for all prime $p$ except from 2. The growth is determined via calculation of Gr\"obner bases and Hilbert series.

Before proceeding, we would like to make one more comment. Ufnarovskii \cite{ufn1,ufn2} demonstrated that a quadratic algebra given by two (homogeneous degree 2) relations is automaton. What one can actually observe  is that it is always an FGB-algebra. On the other hand, one easily checks that any quadratic algebra on two generators is an FGB-algebra as well. Thus, as far as quadratic algebras are concerned, our example is minimal: we could not possibly have less than three generators or less than three relations.

\subsection{Growth of finitely generated algebras}

Part of our argument is based on the growth of finitely generated algebras. We recall relevant concepts. Let $A$ be a finitely generated unital $\K$-algebra and $X$ be a finite generating set for $A$. Let $A^{(n)}=A_X^{(n)}$ be a linear subspace of $A$ spanned by the products of up to $n$ elements of $X$ (we include $1$, which is the product of zero elements of $X$). Clearly, $A^{(n)}$ form a filtration on $A$. The power series $H_A(t)=H^X_A(t)=\sum\limits_{n=0}^\infty a_nt^n$ with $a_0=1$ and $a_n=a_n(X)=\dim A^{(n)}/A^{(n-1)}$ for $n\geq 2$ is called the \textbf{\textit{Hilbert series}} of $A$ (with respect to the generating set $X$). If we know the normal words for $A$ with respect to some well-ordering on $\langle X\rangle$, we know the Hilbert series as well: $a_n$ is the number of normal words of degree $n$. Note that if we assume that each element of $X$ has degree $1$ and that each defining relation is homogeneous, then the ideal of relations $I$ and therefore $A$ itself are degree graded. In this case $a_n$ becomes the dimension of the space $A_n$ of homogeneous elements of $A$ of degree $n$ (by convention, zero has any degree we like). The series $P_A(t)=P^X_A(t)=\frac{H^X_A(t)}{1-t}=\sum\limits_{n=0}^\infty \alpha_nt^n$ with $\alpha_n=\alpha_n(X)=a_0+{\dots}+a_n=\dim A^{(n)}$ is called the \textbf{\textit{Poincar\'e series}} of $A$ (with respect to the generating set $X$).

Let $A$ be a finitely generated unital $\K$-algebra. Obviously, the Hilbert and Poincar\'e series $H^X_A$ and $P^X_A(t)=\sum \alpha_n(X)t^n$ depend on the choice of the generating set $X$. However, some of their features do not depend on this choice. Indeed, let $X$ and $Y$ be two finite generating sets of $A$. Since $X$ and $Y$ are finite, we can choose $m\in\N$ such that $X\subset A_Y^{(m)}$ and $Y\subset A_X^{(m)}$. Then
\begin{equation}\label{albe}
\alpha_n(X)\leq \alpha_{mn}(Y)\ \ \text{and}\ \ \alpha_n(Y)\leq \alpha_{mn}(X)\ \ \text{for all $n\in\Z_+$.}
\end{equation}
In particular, it follows that the limits
\begin{equation}\label{gkd}
\textstyle {\rm GKdim}(A)=\limsup\limits_{n\to\infty}\frac{\ln \alpha_n(X)}{\ln n}\in[0,\infty]
\end{equation}
and
\begin{equation}\label{kappa}
\textstyle \kappa(A)=\limsup\limits_{n\to\infty}\frac{\ln\ln \alpha_n(X)}{\ln n}\in[0,1]
\end{equation}
do not depend on the choice of $X$.

\begin{definition}\label{gek}
The number ${\rm GKdim}(A)$ defined in (\ref{gkd}) is known as the \textbf{\textit{Gelfand--Kirillov dimension}} of $A$. Algebras of finite Gelfand--Kirillov dimension are said to have \textbf{\textit{polynomial growth}}. On the other hand, if
$$
\textstyle \limsup\limits_{n\to\infty}\frac{\ln \alpha_n(X)}{n}>0,
$$
we say that $A$ has \textbf{\textit{exponential growth}}. A finitely generated algebra $A$ which has infinite Gelfand--Kirillov dimension and is not of exponential growth is said to have  \textbf{\textit{intermediate growth}}.
\end{definition}

\begin{remark}\label{gek1}
Again, for an algebra to have exponential growth does not depend on the choice of a finite generating set $X$. Note that while the limit in the above display does depend on $X$, it being positive does not according to (\ref{albe}). Obviously, every algebra with exponential growth has infinite Gelfand--Kirillov dimension. Note also that polynomial and exponential growths are related to the number $\kappa(A)$ defined in (\ref{kappa}):
\begin{equation}\label{poex}
\begin{array}{l}
\text{if $A$ has polynomial growth, then $\kappa(A)=0$;}
\\
\text{if $A$ has exponential growth, then $\kappa(A)=1$.}
\end{array}
\end{equation}
\end{remark}

We use growth arguments in the proof of Theorem~\ref{main}. The second result of this paper is the following example.

\begin{theorem}\label{inter} Let $\K$ be an arbitrary field and $C$ be the $\K$-algebra given by four generators $x,y,z,u$ and seven quadratic relations $xu-yz$, $yu-zx$, $zu-uz$, $yy$, $yx$, $xy$ and $xx$. Then $C$ has intermediate growth.
\end{theorem}

In Section~2 we describe a way to verify that a finitely presented algebra does not possess a finite Gr\"obner basis in its ideal of relations for any choice of generators and ordering. We also do the preparatory work for applying the general result to particular algebras of Theorem~\ref{main}. We study the Gr\"obner bases of members of two varieties of finitely presented algebras containing the algebras $A$ and $B$ of Theorem~\ref{main} in Section~3, which culminates in the proof of Theorem~\ref{main}. We prove Theorem~\ref{inter} in Section~4 and make few extra remarks in the final Section~5.

\section{The IGB-criterion}

Everywhere below we use the following notation. The symbol $\Q$ denotes the field of rational numbers, $\Z$ stands for the ring of integers, $\N$ is the set of positive integers, while $\Z_+=\N\cup\{0\}$ is the set of non-negative integers. If $p$ is a prime number then $\Z_p=\Z/p\Z$ is the $p$-element field. We denote by $\pi_p$ the canonical map $\pi_p:\Z\to\Z_p$: $\pi_p(n)$ is the congruence class of $n$ modulo $p$. We naturally extend $\pi_p$ to a ring homomorphism
$$
\pi_p:\Z\langle x_1,\dots,x_n\rangle \to \Z_p\langle x_1,\dots,x_n\rangle
$$
(we use the same symbol to denote the extension) by acting on coefficients. That is, for $f\in \Z\langle x_1,\dots,x_n\rangle$, $\pi_p(f)$ is obtained from $f$ by replacing each coefficient of $f$ by its congruence class modulo $p$. We would like to introduce the following technical notions.

\begin{definition}\label{frr} Let $\K$ be a field and $R_0$ be a subring of $\K$. We say that $R$ is a \textbf{\textit{finite ring extension of $R_0$ in $\K$}} if $R$ is a subring of $\K$,  $R_0\subseteq R$ and there is a finite set $S\subset R$ such that the set $R_0\cup S$ generates $R$ as a ring.
\end{definition}

\begin{definition}\label{grme} We say that a map $A\mapsto \gamma(A)$, which assigns to a finitely generated algebra $A$ a number $\gamma(A)\in[0,\infty]$ is a \textbf{\textit{growth measuring function}} if the following conditions are satisfied:
\begin{itemize}\itemsep=-3pt
\item $\gamma(A)=\gamma(A_\F)$ for every field $\F$ containing the ground field $\K$ of $A$ as a subfield$;$
\item if $A$ and $B$ are two finitely generated algebras (possibly over different fields) and there exist finite generating sets $X$ and $Y$ for $A$ and $B$ respectively such that the corresponding Poincar\'e series $P_A^X=\sum\alpha_nt^n$ and $P_B^Y=\sum\beta_nt^n$ satisfy $\alpha_n\leq\beta_n$ for all $n\in\Z_+$, then $\gamma(A)\leq\gamma(B)$.
\end{itemize}
\end{definition}

\begin{remark}\label{grme1} Note that the Gelfand--Kirillov dimension ${\rm GKdim}(A)$ as well as $\kappa(A)$ defined in (\ref{kappa}) are growth measuring functions. The dimension $\dim A$ of a finitely generated $\K$-algebra $A$ as a $\K$-vector space is a growth measuring function as well. One can come up with quite a few other examples.
\end{remark}

\begin{theorem}\label{nofgb}
Let $(X,F)$ be a finite presentation of a $\K$-algebra $A$ and let $R_0$ be a unital subring of $\K$ such that each $f\in F$ belongs to $R_0\langle X\rangle$. That is, all coefficients of each $f\in F$ are in $R_0$. Assume also that we have a collection of non-zero ring homomorphisms $\rho_\alpha:R_0\to \K_\alpha$ labelled by $\alpha$ from some index set $\Lambda$ such that each $\K_\alpha$ is a field. We extend every $\rho_\alpha$ to a ring homomorphism $\rho_\alpha:R_0\langle X\rangle\to \K_\alpha\langle X\rangle$ by the original $\rho_\alpha$ acting on coefficients. For each $\alpha\in\Lambda$, let $A_\alpha$ be the $\K_\alpha$-algebra defined by the presentation $(X,F_\alpha)$, where $F_\alpha=\rho_\alpha(F)$. Finally, assume that $\gamma$ is a growth measuring function and the following conditions are satisfied$:$
\begin{itemize}\itemsep=-2pt
\item[\rm (1)] $\gamma(A_\alpha)>\gamma(A)$ for every $\alpha\in\Lambda;$
\item[\rm (2)] for every field $\F$ containing $\K$ as a subfield and for every finite ring extension $R$ of $R_0$ in $\F$, there exists $\alpha\in\Lambda$ such that $\rho_\alpha$ extends to a ring homomorphism $\rho'_\alpha:R\to \F_\alpha$, where $\F_\alpha$ is some field containing $\K_\alpha$ as a subfield.
\end{itemize}
Then $A$ is an IGB-algebra.
\end{theorem}

\begin{proof} Assume the contrary: $A$ is an FGB-algebra. Then there is a field $\F$ containing $\K$ as a subfield such that $A_\F$ has a finite generating set $Y$ and there is a well-ordering on $\langle Y\rangle$ compatible with multiplication for which the ideal of relations of $A_\F$ corresponding to the generating set $Y$ has finite reduced Gr\"obner basis $G=\{g_1,\dots,g_t\}$. If we replace the original field $\K$ by $\F$, nothing really changes except for $A$ being replaced by $A_\F$: $(X,F)$ is still a presentation of $A_\F$ as an $\F$-algebra, $\gamma(A_\F)=\gamma(A)$, $A_\alpha$ do not change at all and (2) remains true. Thus without loss of generality, we can assume that $\F=\K$ and $A_\F=A$. Let $X=\{x_1,\dots,x_n\}$, $Y=\{y_1,\dots,y_m\}$ and $F=\{f_1,\dots,f_s\}$. Since both $X$ and $Y$ are generating sets in $A$, there exist ${\widehat x}_j\in \K\langle t_1,\dots,t_m\rangle$ and ${\widehat y}_k\in \K\langle s_1,\dots,s_n\rangle$ such that $x_j={\widehat x}_j(y_1,\dots,y_m)$ in $A$ for each $1\leq j\leq n$ and $y_k={\widehat y}_k(x_1,\dots,x_n)$ in $A$ for each $1\leq k\leq m$. It follows that
\begin{equation}\label{xj}
x_j{-}{\widehat x}_j({\widehat y}_1(x_1,\dots,x_n),\dots,{\widehat y}_m(x_1,\dots,x_n))=\!\!\sum_i a_{i,j}u_{i,j}f_{r(i,j)}(x_1,\dots,x_n)v_{i,j}\ \text{in $\K\langle X\rangle$}
\end{equation}
for $1\leq j\leq n$, where the sum in the right-hand side is finite, $u_{i,j},v_{i,j}\in \langle X\rangle$ and $a_{i,j}\in\K^*$.

Since each $g_k$ is a relation for $A$ with respect to the generating set $Y$, $g_k(y_1,\dots,y_m)=0$ in $A$ and therefore $g_k({\widehat y}_1(x_1,\dots,x_n),\dots,{\widehat y}_m(x_1,\dots,x_n))=0$ in $A$. This happens precisely when
\begin{equation}\label{phi1}
g_k({\widehat y}_1(x_1,\dots,x_n),\dots,{\widehat y}_m(x_1,\dots,x_n))=\sum_i c_{i,k}q_{i,k}f_{j(i,k)}(x_1,\dots,x_n)w_{i,k}\ \ \text{in $\K\langle X\rangle$},
\end{equation}
for $1\leq k\leq t$, where the sum in the right-hand side is finite, $q_{i,k},w_{i,k}\in \langle X\rangle$ and $c_{i,k}\in\K^*$.

Now consider the finite set $S\subset\K$ comprising all (non-zero) coefficients of all $g_k$, all (non-zero) coefficients of all ${\widehat x}_j$ and all ${\widehat y}_k$, all coefficients $a_{i,j}$ from (\ref{xj}) and all coefficients $c_{i,k}$ from (\ref{phi1}). Let $R$ be the subring of $\K$ generated by $R_0\cup S$. Clearly $R$ is a finite ring extension of $R_0$ in $\K=\F$. By (2), there is $\alpha\in \Lambda$, a field $\F_\alpha$ containing $\K_\alpha$ as a subfield and a ring homomorphism $\rho'_\alpha:R\to \F_\alpha$ such that $\rho'_\alpha\bigr|_{R_0}=\rho_\alpha$. Let $B_\alpha=(A_\alpha)_{\F_\alpha}$. That is, $B_\alpha$ is obtained from $A_\alpha$ by extending the ground field from $\K_\alpha$ to $\F_\alpha$. Hence $\gamma(B_\alpha)=\gamma(A_\alpha)$ and therefore $\gamma(B_\alpha)>\gamma(A)$ by (1). Note that as an $\F_\alpha$-algebra, $B_\alpha$ is presented by $(X,F_\alpha)$.

As we have done with $\rho_\alpha$, we can treat $\rho'_\alpha$ as a homomorphism from $R\langle X\rangle$ to $\F_\alpha\langle X\rangle$. The way we defined $R$ yields that both (\ref{xj}) and (\ref{phi1}) consist of equalities between elements of $R\langle X\rangle$. Since $\rho'_\alpha(f_r)=\rho_\alpha(f_r)$ is a relation for $B_\alpha$ with respect to the generating set $X$, after the application of $\rho'_\alpha$, the right-hand sides of both (\ref{xj}) and (\ref{phi1}) vanish as elements of $B_\alpha$. Consider $\rho'_\alpha({\widehat y}_k)(x_1,\dots,x_n)$ for $1\leq k\leq m$ as elements of $B_\alpha$ and denote them by $y_k$ again. Applying $\rho'_\alpha$ to both sides of (\ref{xj}), we see that
$$
x_j{-}\rho'_\alpha({\widehat x}_j)(y_1,\dots,y_m)=0\ \text{in $B_\alpha$ for $1\leq j\leq n$.}
$$
It follows that $Y=\{y_1,\dots,y_m\}$ is a generating set for $B_\alpha$. Now applying $\rho'_\alpha$ to both sides of (\ref{phi1}), we see that
\begin{equation}\label{asa}
\rho'_\alpha(g_k)(y_1,\dots,y_m)=0\ \text{in $B_\alpha$ for $1\leq k\leq t$.}
\end{equation}
Since $\rho_\alpha$ is non-zero, $R_0$ is unital and $\rho'_\alpha$ extends $\rho_\alpha$, we have $\rho'_\alpha(1)=1$. Keeping $\langle Y\rangle$ equipped with the same ordering $\leq$ for which $G$ is the reduced Gr\"obner basis of the ideal of relations of $A$, we see that the equality  $\rho'_\alpha(1)=1$ implies that the leading monomial of each $g_k$ is the same as the leading monomial of $\rho'_\alpha(g_k)$. By (\ref{asa}), each $\rho'_\alpha(g_k)$ is a relation for $B_\alpha$ with respect to the generating set $Y$. It follows that the normal words for $A$ with respect to $(\langle Y\rangle,\leq)$, span $B_\alpha$ as an $\F_\alpha$-vector space. Hence the Poincar\'e series $P_A=\sum s_nt^n$ of $A$ and $P_{B_\alpha}=\sum r_nt^n$ of $B_\alpha$ with respect to the generating set $Y$ satisfy $r_n\leq s_n$ for all $n\in\Z_+$. Thus $\gamma(B_\alpha)\leq \gamma(A)$, which contradicts the earlier inequality $\gamma(B_\alpha)>\gamma(A)$.
\end{proof}

The only growth measuring function we shall apply Theorem~\ref{nofgb} with is the Gelfand--Kirillov dimension ${\rm GKdim}$. We gave Theorem~\ref{nofgb} in its full generality on the off-chance that using it with other growth measuring functions may happen in the future. In the proof of Theorem~\ref{main}, we are going to apply Theorem~\ref{nofgb} in the particular case $R_0=\Z$. In order to do this smoothly we need the following commutative lemma. We also need two well-known facts. First,
\begin{equation}\label{fgg}
\text{a subgroup of a finitely generated abelian group is finitely generated..}
\end{equation}
This easily follows from the fact that every finitely generated abelian group is isomorphic to a direct product of finitely many cyclic groups. Our second tool is the weak Hilbert Nullstellensatz:
\begin{equation}\label{hns}
\begin{array}l\text{if $\K$ is a field, $\F$ is an algebraically closed field containing $\K$ as a subfield}\\ \text{and $I$ is a proper ideal in the polynomial $\K$-algebra $\K[x_1,\dots,x_n]$,}\\ \text{then there is $(a_1,\dots,a_n)\in\F^n$ such that $f(a_1,\dots,a_n)=0$ for all $f\in I$.}\end{array}
\end{equation}

\begin{lemma}\label{commu} Let $R$ be a non-zero unital commutative finitely generated ring such that the group $(R,+)$ is torsion-free. Let $E$ be the set of all prime numbers $p$ such that $1=py$ for some $y\in R$. Then $E$ is finite.
\end{lemma}

\begin{proof}Let $\{x_1,\dots,x_n\}$ be a generating set for $R$. Then $R$ is naturally isomorphic to $\Z[t_1,\dots,t_n]/I$, where $I$ is the ideal consisting of all $f$ such that $f(x_1,\dots,x_n)=0$ in $R$. Since $(R,+)$ is torsion-free, there is a unique ideal $J$ in the $\Q$-algebra $\Q[t_1,\dots,t_n]$ such that $I=\Z[t_1,\dots,t_n]\cap J$. Since $R$ is unital, $I$ is proper and therefore $J$ is proper. Since the field $\aaa$ of algebraic (over $\Q$) numbers is algebraically closed and contains $\Q$, (\ref{hns}) provides $\alpha_1,\dots,\alpha_n\in\aaa$ such that $f(\alpha_1,\dots,\alpha_n)=0$ for every $f\in J$. Now we have
\begin{equation}\label{one}
p\in E\,\Longrightarrow\, \text{there exists $f\in \Z[t_1,\dots,t_n]$ such that $\textstyle f(\alpha_1,\dots,\alpha_n)=\frac1p$.}
\end{equation}
Pick $k\in\N$ such that $\beta_j=k\alpha_j$ is an algebraic integer for $1\leq j\leq n$. For instance, as $k$ we can take the product of leading coefficients of the minimal polynomials for $\alpha_j$. Obviously,
\begin{equation}\label{two}
\text{for each $f\in \Z[t_1,\dots,t_n]$, there is $g\in \Z[t_1,\dots,t_n]$ such that $\textstyle f(\alpha_1,\dots,\alpha_n)=\frac{g(\beta_1,\dots,\beta_n)}{k^m}$},
\end{equation}
where $m$ is the degree of $f$. Combining (\ref{one}) and (\ref{two}), we get
\begin{equation}\label{three}
p\in E\,\Longrightarrow\, \text{there exist $g\in \Z[t_1,\dots,t_n]$ and $m\in\N$ such that $\textstyle g(\beta_1,\dots,\beta_n)=\frac{k^m}p$.}
\end{equation}

Now consider the set
$$
G=\{g(\beta_1,\dots,\beta_n):g\in \Z[t_1,\dots,t_n]\}.
$$
Obviously, $G$ is a subgroup of $(\aaa,+)$. Since each $\beta_j$ is an algebraic integer, $\beta_j^{m_j}$ is a linear combination with integer coefficients of $1,\beta_j,\dots,\beta_j^{m_j-1}$, where $m_j$ is the degree of $\beta_j$. Hence the set $G$ does not change if in its definition we use only $g\in \Z[t_1,\dots,t_n]$ such that $\deg_{t_j}g<m_j$ for $1\leq j\leq n$. Hence the group $(G,+)$ is finitely generated. By (\ref{fgg}), the group $(G\cap \Q,+)$ is also finitely generated. Since every finitely generated subgroup of $(\Q,+)$ is cyclic, $G\cap \Q$ as an additive group is generated by a fraction $\frac{a}{b}$ with $a\in\Z$, $b\in\N$. By (\ref{three}), for every $p\in E$, there is $r_p\in\Z_+$ such that $\frac{k^{r_p}}{p}\in G\cap Q$. On the other hand, a number of the form $\frac{k^m}{p}$ with $p$ being a prime and $m\in\Z_+$ can only belong to $G\cap Q$ if $p$ divides $N=kb$. Since $N$ has only finitely many prime factors, $E$ is finite.
\end{proof}

\begin{corollary}\label{commu1} Let $R$ be a non-zero unital commutative finitely generated ring such that the group $(R,+)$ is torsion-free. Let $E$ be the set of all prime numbers $p$ such that there is no non-zero ring homomorphisms from $R$ to the algebraic closure $\F_p$ of the field $\Z_p$. Then $E$ is finite.
\end{corollary}

\begin{proof} Let $S=\{s_1,\dots,s_n\}$ be generators of $R$. Consider the ideal $I$ in $\Z[t_1,\dots,t_n]$ consisting of all $f$ such that $f(s_1,\dots,s_n)=0$. It is clear that $p\in E$ if $\pi_p(I)=\Z_p[t_1,\dots,t_n]$. On the other hand, if $\pi_p(I)\neq \Z_p[t_1,\dots,t_n]$, (\ref{hns}) provides $\alpha_1,\dots,\alpha_n\in\F_p$ such that $f(\alpha_1,\dots,\alpha_n)=0$ whenever $f\in I$. Now the map $s_j\mapsto \alpha_j$ extends to a homomorphism from $R$ to $\F_p$ yielding $p\notin E$. Thus $E$ is exactly the set of primes $p$ for which $\pi_p(I)=\Z_p[t_1,\dots,t_n]$. The latter equality is equivalent to $1\in \pi_p(I)$, which in turn, is equivalent to $1=py$ in $R$ for some $y\in R$. By Lemma~\ref{commu}, $E$ is finite.
\end{proof}

Now we state and prove a particular case of Theorem~\ref{nofgb} corresponding to $R_0=\Z$ and the growth measuring function $\gamma$ being the Gelfand--Kirillov dimension ${\rm GKdim}$.

\begin{theorem}\label{fgb} Let $A$ be an algebra over a field $\K$ of characteristic zero given by a finite presentation $(X,F)$ such that each defining relation $(=$an element of $F)$ has integer coefficients. For each prime number $p$, let $A_p$ be the $\Z_p$-algebra defined by the presentation $(X,F_p)$, where $F_p=\pi_p(F)=\{\pi_p(f):f\in F\}$. Assume also that
${\rm GKdim}(A_p)>{\rm GKdim}(A)$ for infinitely many prime numbers $p$. Then $A$ is an IGB-algebra.
\end{theorem}

\begin{proof} We apply Theorem~\ref{nofgb} with $R_0=\Z$ and the homomorphisms $\rho_\alpha$ being $\pi_p:\Z\to\Z_p$ with the index set $\Lambda$ being the infinite set of all prime numbers $p$ for which ${\rm GKdim}(A_p)>{\rm GKdim}(A)$. Thus condition (1) of Theorem~\ref{nofgb} with $\gamma={\rm GKdim}$ is automatically satisfied. In order to complete the proof, it remains to verify condition (2) of  Theorem~\ref{nofgb}. Let $R$ be a finite ring extension of $\Z$ in a field $\F$ containing $\K$ as a subfield. Then $R$ is a finitely generated ring and $(R,+)$ is torsion-free.
Since $\Lambda$ is infinite, Corollary~\ref{commu1} provides $p\in \Lambda$ for which there exists a non-zero ring homomorphism $\phi:R\to\F_p$, where $\F_p$ is the algebraic closure of $\Z_p$. Since $\phi\neq 0$, we must have $\phi(1)=1$, which implies that $\phi$ extends $\pi_p$. Thus condition (2) of Theorem~\ref{nofgb} is satisfied as well and the result follows.
\end{proof}

\section{Two varieties of finitely presented algebras}

In this section we compute reduced Gr\"obner bases of some finitely presented algebras. The algebras we deal with are degree graded. The graded case is much better than the general one in terms of computing the Gr\"obner basis since the elements of the reduced Gr\"obner basis are homogeneous and the degree $n$ elements arise only from resolving the degree $n$ overlaps of the leading monomials (aka obstructions, aka ambiguities) of elements computed so far. Hence the Buchberger algorithm climbs up the degrees in this case: for each given $n$, after finitely many steps we may be sure that no new members of the Gr\"obner basis of degree $n$ will appear. In the non-graded case we do not have even this modest luxury.

\begin{example}\label{exa1} Let $\K$ be a field of characteristic other than $2$, $b\in\K$ is such that $b^2\neq 1$, $b\neq 0$ and $a=\frac{b^2-1}{4}$.. Consider the $\K$-algebra $A$ given by generators $x,y,z$ and relations $xy$, $yz$ and $xx-xz-azz$. We equip monomials with the left-to-right degree-lexicographical ordering assuming $x>y>z$. If $s=\frac{1-b}{1+b}$ has infinite order in the group $\K^*$, then the leading monomials of the members of the reduced Gr\"obner basis of the ideal of relations of $A$ are $yz$, $xz^jx$ and $xz^jy$ with $j\in\Z_+$. If $s$ has finite order $m$ in $\K^*$, then the leading monomials are $yz$, $xz^jx$ and $xz^jy$ with $0\leq j\leq m-2$, $xz^{m}$ and $z^{m}(xz^{m-1})^jy$ for $j\in\Z_+$. Furthermore, $H_A(t)=\frac1{(1-t)^{3}}$ if $s$ has infinite order and $H_{A}(t)=\frac{1-t^{m}-t^{m+1}}{(1-t)^2(1-t-t^{m})}$ if $s$ has order $m\in\N$.
\end{example}

\begin{proof} Assume that $k\in\N$ and $s^j\neq 1$ for $1\leq j\leq k$. A direct computation shows that the  $yz$ together with
$$
\textstyle p_{j+1}xz^jx-\frac{p_{j+2}}{2}xz^{j+1}+2ap_jz^{j+1}x-ap_{j+1}z^{j+2},\ \ p_{j+1}xz^jy+2ap_jz^{j+1}y\ \ \text{for $0\leq j\leq k$,}
$$
where $p_j=(1{+}b)^j-(1{-}b)^j$, comprise (up to normalization) all members of degree up to $k+2$ of the reduced Gr\"obner basis of the ideal of relations of $A$. Indeed, relations in the above display for $j=0$ are (up to a non-zero scalar multiple) the defining relations $x^2-xz-az^2$ and $xy$ and we proceed inductively. The degree $j+2$ overlaps $xz^{j-1}x^2=(xz^{j-1}x)x=xz^{j-1}(xx)$ and\break $xz^{j-1}xy=(xz^{j-1}x)y=xz^{j-1}(xy)$ produce (after reduction) the two relations in the above display, while all other overlaps of degree $j+2$ resolve.

If $s$ has infinite order, this works for every $k$ and we are furnished with the complete reduced Gr\"obner basis. The leading monomials of the members of the basis are $yz$, $xz^mx$ and $xz^my$ for $m\in\Z_+$. The corresponding normal words are $z^jy^m$ and $z^jy^mxz^n$ with $j,m,n\in\Z_+$. Counting the number of normal words of any given degree, we arrive to $H_{A}(t)=\frac1{(1-t)^{3}}$.

Now assume that the order of $s$ in $\K^*$ is $k+1$ with $k\in\N$ (note that $s\neq 1$ since $b\neq 0$). The above display together with $yz$ still provides all members of degree up to $k+2$ of the reduced Gr\"obner basis of the ideal of relations of $A$. However, the degree $k+2$ elements are now up to a scalar multiple $z^{k+1}y$ and $xz^{k+1}-z^{k+1}x$. From this point on the pattern of the basis changes dramatically. The overlaps $xz^{2k+1}(xz^k)^jy=[xz^{k+1}]z^k(xz^k)^jy=xz^k[z^{k+1}(xz^k)^jy]$ yield the monomial relation $z^{k+1}(xz^k)^{j+1}y$ for every $j\in\Z_+$ while all other overlaps of degrees $\geq k+3$ resolve. Thus the leading monomials of the members of the reduced Gr\"obner basis of the ideal of relations of $A$ are $yz$, $xz^jx$ and $xz^jy$ with $0\leq j\leq m-2$, $xz^{m}$ and $z^{m}(xz^{m-1})^jy$ for $j\in\Z_+$, where $m=k+1$ is the order of $s$. The corresponding normal words are $z^j$ with $j\in\Z_+$, $z^jy^r$ with $j,r\in\Z_+$, $j<m$, $z^r(xz^{m-1})^jxz^t$ with $r,j,t\in\Z_+$, $t<m$, $r\geq m$ and $z^ry^{m_0}xz^{m-1}y^{m_1}xz^{m-1}y^{m_2}x\dots xz^{m-1}y^{m_j}$, $z^ry^{m_0}xz^{m-1}y^{m_1}xz^{m-1}y^{m_2}x\dots xz^{m-1}y^{m_j}xz^t$ with $j,r,t,m_0,\dots,m_j\in\Z_+$, $r,t<m$. Counting the number of normal words of given degree, we arrive to $H_{A}=\frac{1-t^{m}-t^{m+1}}{(1-t)^2(1-t-t^{m})}$.
\end{proof}

\begin{example}\label{exa2} Let $\K$ be an arbitrary field and $a\in\K^*$. Consider the $\K$-algebra $B$ given by generators $x,y$ and relations $y^3$ and $x^2y{-}ayx^2{-}yxy$. We equip monomials with the left-to-right degree-lexicographical ordering assuming $x>y$. If $1{+}a{+}\dots{+}a^k\neq 0$ for all $k\in\N$, then the leading monomials of the members of the reduced Gr\"obner basis of the ideal of relations of $B$ are $y^3$, $x^2y$ and $y^2(xy)^jxy^2$ with $j\in\Z_+$. If $k\in\N$ is the minimal positive integer for which $1{+}a{+}\dots{+}a^k=0$, then the leading monomials are $y^3$, $x^2y$ and $y^2(xy)^jxy^2$ with $0\leq j<k$ $($finite Gr\"obner basis in this case$)$. In the first case $H_B(t)=\frac{1}{(1+t)(1-t)^{3}}$, while in the second case $H_B(t)=\frac{1-t^{2k+3}}{(1-t)^2(1-t^2-t^{2k+3})}$.
\end{example}

\begin{proof} Whatever the case, a direct computation shows that apart for defining relations the only other degree $\leq 5$ member of the reduced Gr\"obner basis is the monomial $y^2xy^2$, which arises from the overlap $x^2y^3=(x^2y)y^2=x^2(y^3)$. If $1{+}a=0$, the defining relations together with $y^2xy^2$ form the reduced Gr\"obner basis. If $1{+}a\neq 0$, apart from the above three members, the only degree $\leq 7$ member of the reduced Gr\"obner basis is the monomial $y^2xyxy^2$, which arises from the overlap $x^2y^2xy^2=(x^2y)yxy^2=x^2(y^2xy^2)$. If $1{+}a{+}a^2=0$, the process ends (we have the complete basis). Otherwise, apart from the above four members, the only degree $\leq 9$ member of the reduced Gr\"obner basis is the monomial $y^2xyxyxy^2$, which arises from the overlap $x^2y^2xyxy^2=(x^2y)yxyxy^2=x^2(y^2xyxy^2)$. This pattern goes on indefinitely (one can give an easy inductive argument as well). This takes care of the Gr\"obner bases in all cases and of their leading monomials.

Now if $1{+}a{+}\dots{+}a^k\neq 0$ for all $k\in\N$, then as we have just seen, the leading monomials of the members of the reduced Gr\"obner basis of the ideal of relations of $B$ are $y^3$, $x^2y$ and $y^2(xy)^jxy^2$ with $j\in\Z_+$. Hence the normal words of $B$ are $y^\epsilon(xy)^mxy^2(xy)^jx^s$ with $\epsilon\in\{0,1\}$, $m,j,s\in\Z_+$ and $y^{\rm GKdim}(xy)^jx^s$ with ${\rm GKdim}\in\{0,1,2\}$ and $j,s\in\Z_+$. Counting normal words of given degree yields $H_B(t)=\frac{1}{(1+t)(1-t)^{3}}$.

If $k\in\N$ is the minimal positive integer for which $1{+}a{+}\dots{+}a^k=0$, then as we have already seen, the leading monomials of the members of the reduced Gr\"obner basis of the ideal of relations of $B$ are $y^3$, $x^2y$ and $y^2(xy)^jxy^2$ with $0\leq j<k$. Now we have a finite Gr\"obner basis and applying standard techniques \cite{ufn1}, one gets $H_B(t)=\frac{1-t^{2k+3}}{(1-t)^2(1-t^2-t^{2k+3})}$.
\end{proof}

Now we plug $a=2$ into Example~\ref{exa1}.

\begin{lemma}\label{exale} Let $\K$ be a field and $A$ be the quadratic $\K$-algebra given by generators $x,y,z$ and relations $xy$, $yz$ and $x^2{-}xz{-}2z^2$. If ${\rm char}\,\K=0$, then $A$ is an automaton algebra and satisfies ${\rm GKdim}(A)=3$. If ${\rm char}\,\K$ is a prime $p\geq 5$, then $A$ has exponential growth and therefore ${\rm GKdim}(A)=\infty$.
\end{lemma}

\begin{proof} We equip monomials with the left-to-right degree-lexicographical ordering assuming $x>y>z$. If ${\rm char}\,\K\notin\{2,3\}$, $A$ is the algebra from Example~\ref{exa1} with $a=2$, $b=-3$ and $s=-2$. If ${\rm char}\,\K=0$, $s$ has infinite order in $\K^*$. By Example~\ref{exa1}, $H_A(t)=\frac1{(1-t)^{3}}$ (hence ${\rm GKdim}(A)=3$) and the leading monomials of the members of the reduced Gr\"obner basis of the ideal of relations of $A$ are $yz$, $xz^jx$ and $xz^jy$ with $j\in\Z_+$. They are easily seen to form a regular language. Indeed, the corresponding graph is
$$
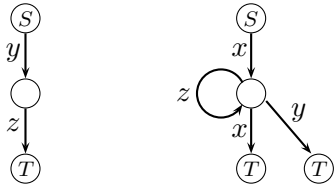
\begin{pspicture}(0,0)(6,3)
\pscircle[linewidth=.4pt](1,2.5){0.2}
\put(0.9,2.4){$\scriptstyle S$}
\psline[linewidth=.8pt]{->}(1,2.3)(1,1.7)
\put(0.8,2){$\!y$}
\pscircle[linewidth=.4pt](1,1.5){0.2}
\psline[linewidth=.8pt]{->}(1,1.3)(1,0.7)
\pscircle[linewidth=.4pt](1,0.5){0.2}
\put(0.8,1){$\!z$}
\put(0.9,0.4){$\scriptstyle T$}
\pscircle[linewidth=.4pt](4,2.5){0.2}
\put(3.9,2.4){$\scriptstyle S$}
\psline[linewidth=.8pt]{->}(4,2.3)(4,1.7)
\put(3.8,2){$\!x$}
\pscircle[linewidth=.4pt](4,1.5){0.2}
\psarc[linewidth=.8pt]{->}(3.6,1.5){0.32}{27}{333}
\put(3,1.4){$z$}
\psline[linewidth=.8pt]{->}(4,1.3)(4,0.7)
\pscircle[linewidth=.4pt](4,0.5){0.2}
\put(3.8,0.92){$\!x$}
\put(3.9,0.4){$\scriptstyle T$}
\psline[linewidth=.8pt]{->}(4.2,1.4)(4.8,0.7)
\pscircle[linewidth=.4pt](4.9,0.5){0.2}
\put(4.8,0.4){$\scriptstyle T$}
\put(4.6,1.2){$\!y$}
\end{pspicture}
$$
Hence $A$ is an automaton algebra and ${\rm GKdim}(A)=3$ provided ${\rm char}\,\K=0$. If ${\rm char}\,\K=p\geq 5$, then by Example~\ref{exa1}, $H_A(t)=\frac{1-t^{m}-t^{m+1}}{(1-t)^2(1-t-t^{m})}$, where $m=m(p)$  is the order of $-2$ in $\Z_p^*$ if $p\geq 5$. The latter shows that $A$ has exponential growth. Indeed, the Poincar\'e series of $A$ is $P_A(t)=\frac{1-t^{m}-t^{m+1}}{(1-t)^3(1-t-t^{m})}=\sum r(n)t^n$ and an easy calculus exercise shows that $\lim\limits_{n\to\infty}\frac{\ln r(n)}{n}=\ln\frac1{c_m}>0$, where $c_m$ is the only zero of $1-t-t^m$ in the interval $(0,1)$.
\end{proof}

Next we plug in $a=1$ into Example~\ref{exa2}.

\begin{lemma}\label{exale1} Let $\K$ be a field and $B$ be the cubic $\K$-algebra given by generators $x,y$ and relations $y^3$ and $x^2y-yx^2-yxy$. If ${\rm char}\,\K=0$, then $B$ is an automaton algebra and satisfies ${\rm GKdim}(B)=3$. If ${\rm char}\,\K\neq 0$, then $B$ has exponential growth and therefore ${\rm GKdim}(B)=\infty$.
\end{lemma}

\begin{proof} We equip monomials with the left-to-right degree-lexicographical ordering assuming $x>y$. If ${\rm char}\,\K=0$, Example~\ref{exa2} yields that $H_B(t)=\frac1{(1+t)(1-t)^{3}}$ (hence ${\rm GKdim}(B)=3$) and that the leading monomials of the members of the reduced Gr\"obner basis of the ideal of relations of $B$ are $y^3$, $x^2y$ and $y^2(xy)^jxy^2$ with $j\in\Z_+$. They are easily seen to form a regular language. Indeed, the corresponding graph is
$$
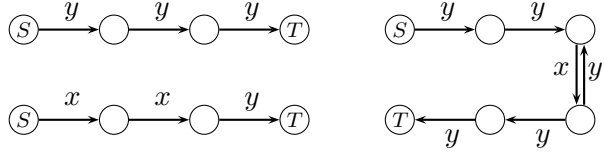
\begin{pspicture}(0,0)(12,3)
\pscircle[linewidth=.4pt](1,2.5){0.2}
\put(0.9,2.4){$\scriptstyle S$}
\psline[linewidth=.8pt]{->}(1.2,2.5)(2,2.5)
\put(1.6,2.7){$\!y$}
\pscircle[linewidth=.4pt](2.2,2.5){0.2}
\psline[linewidth=.8pt]{->}(2.4,2.5)(3.2,2.5)
\put(2.8,2.7){$\!y$}
\pscircle[linewidth=.4pt](3.4,2.5){0.2}
\psline[linewidth=.8pt]{->}(3.6,2.5)(4.4,2.5)
\put(4,2.7){$\!y$}
\pscircle[linewidth=.4pt](4.6,2.5){0.2}
\put(4.5,2.4){$\scriptstyle T$}
\pscircle[linewidth=.4pt](1,1.3){0.2}
\put(0.9,1.2){$\scriptstyle S$}
\psline[linewidth=.8pt]{->}(1.2,1.3)(2,1.3)
\put(1.6,1.5){$\!x$}
\pscircle[linewidth=.4pt](2.2,1.3){0.2}
\psline[linewidth=.8pt]{->}(2.4,1.3)(3.2,1.3)
\put(2.8,1.5){$\!x$}
\pscircle[linewidth=.4pt](3.4,1.3){0.2}
\psline[linewidth=.8pt]{->}(3.6,1.3)(4.4,1.3)
\put(4,1.5){$\!y$}
\pscircle[linewidth=.4pt](4.6,1.3){0.2}
\put(4.5,1.2){$\scriptstyle T$}
\pscircle[linewidth=.4pt](6,2.5){0.2}
\put(5.9,2.4){$\scriptstyle S$}
\psline[linewidth=.8pt]{->}(6.2,2.5)(7,2.5)
\put(6.6,2.7){$\!y$}
\pscircle[linewidth=.4pt](7.2,2.5){0.2}
\psline[linewidth=.8pt]{->}(7.4,2.5)(8.2,2.5)
\put(7.8,2.7){$\!y$}
\pscircle[linewidth=.4pt](8.4,2.5){0.2}
\psline[linewidth=.8pt]{->}(8.35,2.3)(8.35,1.5)
\psline[linewidth=.8pt]{<-}(8.45,2.3)(8.45,1.5)
\put(8.1,1.9){$\!x$}
\put(8.5,1.9){$y$}
\pscircle[linewidth=.4pt](8.4,1.3){0.2}
\psline[linewidth=.8pt]{->}(8.2,1.3)(7.4,1.3)
\pscircle[linewidth=.4pt](7.2,1.3){0.2}
\psline[linewidth=.8pt]{->}(7,1.3)(6.2,1.3)
\pscircle[linewidth=.4pt](6,1.3){0.2}
\put(7.8,1){$y$}
\put(6.6,1){$y$}
\put(5.9,1.2){$\scriptstyle T$}
\end{pspicture}
$$
Hence $B$ is an automaton algebra and ${\rm GKdim}(B)=3$ provided ${\rm char}\,\K=0$. If ${\rm char}\,\K=p$ with $p$ being a prime number, then $k$ in Example~\ref{exa1} equals $p-1$ and we have $H_B(t)=\frac{1-t^{2p+2}}{(1-t)^2(1-t^2-t^{2p+2})}$. This shows that $B$ has exponential growth. Indeed, the Poincar\'e series of $B$ is  $P_B(t)=\frac{1-t^{2p+2}}{(1-t)^3(1-t^2-t^{2p+2})}=\sum r(n)t^n$ and an easy calculus exercise shows that $\lim\limits_{n\to\infty}\frac{\ln r(n)}{n}=\ln\frac1{c_p}>0$, where $c_p$ is the only zero of the polynomial $1-t^2-t^{2p+2}$ in the interval $(0,1)$.
\end{proof}

\subsection{Proof Theorem~\ref{main}}

The combination of Lemmas~\ref{exale}, \ref{exale1} and Theorem~\ref{fgb} yield Theorem~\ref{main}. Indeed, by Lemmas~\ref{exale} and \ref{exale1}, $A$ and $B$ from Theorem~\ref{main} are automaton algebras. The same lemmas also imply that ${\rm GKdim}(A)={\rm GKdim}(B)=3$, while ${\rm GKdim}(A_p)={\rm GKdim}(B_p)=\infty$ for every prime number $p\geq5$. By Theorem~\ref{fgb} then $A$ and $B$ are IGB-algebras, which completes the proof of Theorem~\ref{main}.

\section{Proof of Theorem~\ref{inter}}

Let $\K$ be an arbitrary field and $C$ be the $\K$-algebra given by four generators $x,y,z,u$ and seven quadratic relations $xu-yz$, $yu-zx$, $zu-uz$, $yy$, $yx$, $xy$ and $xx$. We have to show that $C$ has intermediate growth. We equip the monomials in $x,y,z,u$ with the left-to-right degree-lexicographical ordering assuming $x>y>z>u$. First, we find the reduced Gr\"obner basis in the ideal of relations of $C$. Degree 3 overlaps $xyu=(xy)u=x(yu)$ and $yyu=(yy)u=y(yu)$ produce two monomial relations $yzx$ and $xzx$, while all other degree 3 overlaps of the leading monomials of defining relations resolve. Thus we have exactly two degree three members of the reduced Gr\"obner basis: $yzx$ and $xzx$. Degree 4 overlaps $yzxu=(yzx)u=yz(xu)$ and $xzxu=(xzx)u=xz(xu)$ produce two monomial relations $yzyz$ and $xzxz$, while all other degree 4 overlaps resolve. Thus we have exactly two degree 4 members of the reduced Gr\"obner basis: $yzyz$ and $xzxz$. Degree 5 overlaps $yzyzu=(yzyz)u=yzy(zu)$ and $xzxzu=(xzxz)u=xzx(zu)$ produce (after reduction) two monomial relations $yz^2xz$ and $xz^2xz$, while all other degree 5 overlaps resolve. Thus we have exactly two degree 5 members of the reduced Gr\"obner basis: $yz^2xz$ and $xz^2xz$. This pattern is easily seen to go on:
\begin{equation}\label{rgbC}
\begin{array}{l}
\text{the reduced  Gr\"obner basis in the ideal of relations of $C$ consists of}\\ \text{$xu{-}yz$, $yu{-}zx$, $zu{-}uz$, $yx$, $xx$, $yz^{k+1}xz^k$, $xz^{k+1}xz^k$, $yz^kyz^k$ and $xz^kyz^k$ for $k\in\Z_+$.}
\end{array}
\end{equation}

Indeed, one can use the following inductive argument. If $n\geq 5$ is even, we have $n=2k+2$ with $k>1$. The degree $n+1$ overlaps $yz^kyz^{k}u=(yz^kyz^{k})u=yz^kyz^{k-1}(zu)$ and $xz^kyz^{k}u=(xz^kyz^{k})u=xz^kyz^{k-1}(zu)$ produce (after reduction) $yz^{k+1}xz^{k}$ and $xz^{k+1}xz^{k}$. All other overlaps of degree $n+1$ resolve. It is especially easy to see since an overlap of two monomial relations always resolves. If $n\geq 5$ is odd, we have $n=2k+1$ with $k > 1$. The degree $n+1$ overlaps $yz^{k}xz^{k-1}u=(yz^{k}xz^{k-1})u=yz^{k}xz^{k-2}(zu)$ and $xz^{k}xz^{k-1}u=(xz^{k}xz^{k-1})u=xz^{k}xz^{k-2}(zu)$ produce (after reduction) $yz^{k}yz^{k}$ and $xz^{k}yz^{k}$. All other overlaps of degree $n+1$ resolve. This argument verifies (\ref{rgbC}), from which we immediately have
\begin{equation}\label{rgbC1}
\begin{array}{l}
\text{the leading monomials of the reduced  Gr\"obner basis in the ideal of relations of $C$ are}\\ \text{$xu$, $yu$, $zu$, $yx$, $xx$, $yz^{k+1}xz^k$, $xz^{k+1}xz^k$, $yz^kyz^k$ and $xz^kyz^k$ for $k\in\Z_+$.}
\end{array}
\end{equation}
Now we can easily describe the corresponding normal words:
\begin{equation}\label{rgbC2}
\begin{array}{l}
\text{the normal words for $C$ are $u^jz^k$ and $u^jz^{k}s_0z^{k_1}s_1z^{k_2}s_2{\dots}z^{k_m}s_mz^{k_{m+1}}$,}\\ \text{where $s_r\in\{x,y\}$, $j,k,m,k_0,\dots,k_{m+1}\in\Z_+$, $k_1>k_{2}>{\dots}>k_{m+1}$}\\ \text{and $k_r>k_{r+1}+1$ whenever $1\leq r\leq m$ and $s_r=x$}.
\end{array}
\end{equation}
Let $p(n)$ be the number of normal words for $C$ of degree at most $n$. That is, the Poincar\'e series of $C$ is $P_C(t)=\sum p(n)t^n$. According to (\ref{poex}), in order to show that $C$ has intermediate growth, it suffices to prove that
\begin{equation}\label{trtt}
\textstyle \lim\limits_{n\to\infty} \frac{\ln\ln p(n)}{\ln n}=\frac12.
\end{equation}

Let $n\in\N$ and $m$ be the largest integer not exceeding $\sqrt n$. By (\ref{rgbC2}), the words
$$
\text{$s_1z^{2m-2}s_2z^{2m-4}s_3z^{2m-6}\dots z^4s_{m-1}z^2s_m$ with $s_j\in\{x,y\}$}
$$
are normal words for $C$ of degree $\leq n$. Since we listed $2^m$ words, we have $p(n)\geq 2^m$. Since $\sqrt n-1<m\leq \sqrt n$, it follows that
\begin{equation}\label{trtt1}
\textstyle \liminf\limits_{n\to\infty}\frac{\ln p(n)}{\sqrt n}\geq \ln 2\ \Longrightarrow\ \liminf\limits_{n\to\infty} \frac{\ln\ln p(n)}{\ln n}\geq \frac12.
\end{equation}

Next, by (\ref{rgbC2}), the number $a(n)$ of normal words of $C$ of degree $n$ increases as a function of $n$. Since $p(n)=a(0)+{\dots}+a(n)$, we have $p(n)\leq (n{+}1)a(n)$. If $b(n)$ is the number of normal words for $C$ of degree $n$ starting with $x$ or $y$ (we include the empty word $1$ as well), (\ref{rgbC2}) implies that $a(n)=(n{+}1)b(0)+nb(1)+{\dots}+2b(n{-}1)+b(n)$. Since $b(n)$ increases as well, $a(n)\leq \frac{n(n{+}1)}{2}b(n)\leq (n{+}1)^2b(n)$. Hence $p(n)\leq (n{+}1)a(n)\leq (n{+}1)^3b(n)$. By  (\ref{rgbC2}), $b(n)$ does not exceed the number $\phi(n)$ of degree $n$ words of the form $s_1z^{k_1}s_2z^{k_2}s_3{\dots}z^{k_{m-1}}s_mz^{k_{m}}$, where $s_j\in\{x,y\}$, $m\in\N$, $k_1,\dots,k_{m}\in\Z_+$ and $k_1>k_{2}>{\dots}>k_m$. Now one can easily see that the numbers $\phi(n)$ have the following generating function:
\begin{equation}\label{gen}
\textstyle \sum\limits_{n=0}^\infty \phi(n)t^n=\prod\limits_{n=1}^\infty (1+2t^n).
\end{equation}
Note that the infinite product in the right-hand side of (\ref{gen}) converges uniformly on compact subsets of the open unit disk ${\mathbb D}$ of the complex plane $\C$. Thus we can treat it as a holomorphic function $\Phi$ on ${\mathbb D}$. By the Cauchy formula
\begin{equation}\label{cau}
\phi(n)=\frac{1}{2\pi i}\oint_{\Gamma} \frac{\Phi(z)\,dz}{z^{n+1}},
\end{equation}
where $\Gamma$ is any closed contour inside ${\mathbb D}$ which encircles $0$ once counterclockwise. Now we specify $\Gamma$ to be the circle $\Gamma_n$ centered at zero of radius $r_n=1-\frac1{\sqrt n}$ (assume $n\geq 2$). First, observe that $|\Phi(z)|\leq \prod\limits_{k=1}^\infty (1+2r_n^k)$ for $z\in\Gamma_n$. Hence,
$$
\textstyle \ln|\Phi(z)|\leq \sum\limits_{k=1}^\infty \ln(1+2r_n^k)\leq 2\sum\limits_{k=1}^\infty r_n^k=\frac{2r_n}{1-r_n}\leq 2\sqrt n\ \ \text{for $z\in\Gamma_n$}.
$$
Hence
$$
|\Phi(z)|\leq e^{2\sqrt n}\ \ \text{for $z\in\Gamma_n$}.
$$
Using the L'Hopital rule or the Taylor expansion of $\ln(1+t)$, one easily shows that
$$
\textstyle\lim\limits_{n\to\infty} \frac{r_n^{-n}}{e^{\sqrt n}}=\sqrt e.
$$
Using the above two displays and (\ref{cau}) with $\Gamma=\Gamma_n$ and estimating the integral from above by the product of the length of the path by the maximum of the absolute value of the function under the integral, we see that there exists a constant $a>0$ such that
$$
\phi(n)=|\phi(n)|\leq r_n^{-n}\max\{|\Phi(z)|:z\in\Gamma_n\}\leq ae^{3\sqrt n}\ \ \text{for all $n$}.
$$
Since $p(n)\leq (n{+}1)a(n)\leq (n{+}1)^3b(n)\leq (n{+}1)^3\phi(n)\leq a(n{+}1)^3e^{3\sqrt n}$, we have
\begin{equation}\label{trtt2}
\textstyle  \limsup\limits_{n\to\infty}\frac{\ln p(n)}{\sqrt n}\leq e^3\ \Longrightarrow\ \limsup\limits_{n\to\infty} \frac{\ln\ln p(n)}{\ln n}\leq \frac12.
\end{equation}
The inequalities (\ref{trtt1}) and (\ref{trtt2}) imply (\ref{trtt}), which completes the proof of Theorem~\ref{inter}.

\section{Concluding remarks}


\begin{remark}\label{rem2} One could apply similar arguments starting with other than in Examples~\ref{exa1} or~\ref{exa2} varieties of finitely presented algebras. What one needs is many (impossible to cover by a single non-trivial Zarisski closed set) algebras with exceptionally fast growth rate compared to that of generic members of the variety. We went for the ones we considered to be the simplest among other suitable varieties.
\end{remark}

\begin{remark}\label{rem3} Since the algebra $A$ in Theorem~\ref{main} is quadratic, one may ask whether it is Koszul and whether Koszulity has any bearing on the question we considered. We refer to the book \cite{popo} for the definition and the properties of Koszul algebras. Well, $A$ is Koszul. This can be verified directly by computing the Koszul complex of $A$ and showing that it is exact. Thus Koszulity does not save the day: a Koszul automaton algebra can still be an IGB-algebra.
\end{remark}

\begin{remark}\label{rem4} It so happens that all examples of automaton IGB algebras we were able to produce (including those we did not mention in this text) have non-trivial zero divisors. We have no idea of why it happens this way. This prompts us to ask the following question. Unfortunately, we have no arguments to back either positive or negative answer.
\end{remark}

\begin{question}\label{q1} Does there exist an automaton IGB-algebra $A$, which is also a domain $(=$has no non-trivial zero divisors$)?$
\end{question}

\begin{remark}\label{rem5} Another reason we include Theorem~\ref{inter} is that, as it is apparent from its proof, the ideal of relations of the algebra $C$ with respect to the original generators and the left-to-right degree-lexicographical ordering assuming $x>y>z>u$ is very nice (although infinite) in a sense that it follows a clear pattern and can be written in full. On the other hand, the algebra $C$ is not an automaton algebra because it has intermediate growth, see \cite{ufn1}.
\end{remark}

\begin{remark}\label{rem6} More subtle considerations than in the proof of Theorem~\ref{inter} allow to show that the sequence $\frac{\ln p(n)}{\sqrt n}$ converges to a positive limit. What we actually did in the proof is to demonstrate that its lower limit is at least $\ln2$ and its upper limit is at most $e^3$.
\end{remark}

\begin{remark}\label{rem7} As we have mentioned in the introduction, a quadratic algebra with either at most two generators or at most two relations is an FGB-algebra and therefore can not have intermediate growth. Thus a quadratic algebra of intermediate growth has at least 3 generators and at least 3 linearly independent quadratic relations. Our example in Theorem~\ref{inter} has 4 generators and 7 quadratic relations. So, potentially, there is still room for improvement. In particular, the following question is open.
\end{remark}

\begin{question}\label{q2} Does there exist a quadratic algebra $A$ of intermediate growth presented by three generators and some homogeneous degree $2$ relations$?$
\end{question}

We do have an argument showing that if such an algebra $A$ exists, its space of quadratic relations must have dimension either $4$ or $5$. That is, $A$ must be presented by 3 generators and either 4 or 5 linearly independent quadratic relations.


\medskip

{\bf Acknowledgements:} \

We are grateful to IHES and MPIM for hospitality, support, and excellent research atmosphere.
This work is funded by the ERC grant 320974, and partially supported by the project PUT9038. We would also like to express our gratitude
to the anonymous referee for numerous comments, which helped to improve the presentation of the paper.


\vfill\break

\normalsize

\vskip1truecm

\scshape

\noindent  Natalia Iyudu\

\noindent School of Mathematics

\noindent  The University of Edinburgh

\noindent James Clerk Maxwell Building

\noindent The King's Buildings

\noindent Peter Guthrie Tait Road

\noindent Edinburgh

\noindent Scotland EH9 3FD

\noindent E-mail addresses: \qquad {\tt niyudu@exseed.ed.ac.uk}

\vskip1truecm

\scshape

\noindent    Stanislav Shkarin\

\noindent Queens's University Belfast

\noindent Department of Pure Mathematics

\noindent University road, Belfast, BT7 1NN, UK

\noindent E-mail addresses: \qquad  {\tt s.shkarin@qub.ac.uk}


\begin{thebibliography}{99}

\itemsep=-2pt

\bibitem{BBL}A.~Belov, V.~Borisenko and V.~Latyshev, \it Monomial algebras, \rm VINITI, Modern Math. and Appl, {\bf 26}, Algebra 2, (2002), 35--2014

\bibitem{ber}G.~Bergman, \it The diamond lemma for ring theory, \rm Adv. Math. \bf 29\rm\ (1978), 178--218

\bibitem{aut1}F.~Ced\'o and J.~Okni\'nski, \it Gr\"obner bases for quadratic algebras of skew type, \rm Proc. Edinburgh Math. Soc. {\bf 55}\ (2012), 387-–401

\bibitem{gov}V.~Govorov, \it Graded algebras, \rm Math. Notes \bf 12\rm\ (1972), 552--556

\bibitem{kkm}A.~Kirillov, M.~Kontsevich and A.~Molev, \it Algebras of intermediate growth,\rm\ Selecta Mathematica Sovietica {\bf 9}\ (1990), 137--153

\bibitem{koba}Y.~Kobayashi, \it A finitely presented monoid which has solvable word problem but has no regular complete presentation, \rm Theoret. Comput. Sci. {\bf 146} (1995), 321--329

\bibitem{ko}D.~Ko\c cak, \it Finitely presented quadratic algebras of intermediate growth, \rm Algebra and Discrete Mathematics {\bf 20}\ (2015), 69--88

    \bibitem{L}M.~Lothaire, \it Combinatorics on Words, \rm Encyclopedia of Mathematics and its Applications, Addison-Wesley, Reading, MA, {\bf 17}, (1983)

\bibitem{MN}J.~M\aa nsson and P.~Nordbeck, \it Regular Gr\"obner bases, \rm J.. Symb. Comp. \bf 33\rm\ (2002), 163--181

\bibitem{MN1}J.~M\aa nsson and P.~Nordbeck, \it A generalized Ufnarovski graph, \rm Applicable Algebra in Engineering, Communication and Computing {\bf 16}\ (2005), 293--306

\bibitem{popo}A.~Polishchuk and L.~Positselski, \it Quadratic
algebras, \rm University Lecture Series \bf37\rm\ American
Mathematical Society, Providence, RI, 2005

\bibitem{S}M.~Sapir, \it Combinatorial algebra: syntax and semantics, \rm Springer, 2014.

\bibitem{she}J.~Shearer, \it A graded algebra with a non-rational Hilbert series, \rm J. Algebra \bf62\rm\ (1980), 228--231

\bibitem{ufn1}V.~Ufnarovski, \it Combinatorial and asymptotic methods in algebra, \rm Algebra {\bf VI}, Enciclopaedia Math. Sci. {\bf 57}, Springer, Berlin, 1995, 1--196

\bibitem{ufn2}V.~Ufnarovski, \it On the use of graphs for calculating the basis, growth and Hilbert series of associative algebras, \rm Math. Sb. \bf180\rm\ (1989), 1548--1560

\bibitem{ufn3}V.~Ufnarovski, \it Poincar\'e series of graded algebras, \rm Math. Notes {\bf 27} (1980), 12–-18

\end{thebibliography}
\end{document}